\documentclass{amsart}

\usepackage{amsmath}
\usepackage{graphicx}
\usepackage{amsfonts}
\usepackage{amssymb}
\usepackage{amsthm}

\theoremstyle{plain}
\newtheorem{theorem}{Theorem}
\newtheorem{corollary}{Corollary}
\newtheorem{lemma}{Lemma}
\numberwithin{equation}{section}
\theoremstyle{definition}
\newtheorem{definition}{Definition}
\theoremstyle{remark}
\newtheorem{remark}{Remark}

\def\theenumi{\arabic{enumi}}

\def\theenumii{\alph{enumii}}
\def\p@enumii{\theenumi.}

\def\theenumiii{\arabic{enumiii}}
\def\p@enumiii{(\theenumi)(\theenumii)}

\def\p@enumiv{\p@enumiii.\theenumiii}
\newcommand{\N}{\mathbb{N}}

\newcommand{\mex}[1]{\text{mex}{ #1}}

\begin{document}

\title[A Subtraction Game on Graphs]{Combinatorial Analysis of a Subtraction Game on Graphs}
\author{Richard Adams}
\address{Department of Mathematics. California State University, Fresno. Fresno, CA 93740.}
\email{integragsr2001@mail.fresnostate.edu}
\author{Janae Dixon}
\address{Department of Mathematics. California State University, Fresno. Fresno, CA 93740.}
\email{janaekeys@mail.fresnostate.edu}
\author{Jennifer Elder}
\address{Department of Mathematics. California State University, Fresno. Fresno, CA 93740.}
\email{arwenu@mail.fresnostate.edu}
\author{Jamie Peabody}
\address{Department of Mathematics. Kansas State University. Manhattan, KS 66506.}
\email{jpeabody@math.ksu.edu}
\author{Oscar Vega}
\address{Department of Mathematics. California State University, Fresno. Fresno, CA 93740.}
\email{ovega@csufresno.edu}
\author{Karen Willis}
\address{Clovis Community College. Fresno, CA 93730.}
\email{karen.willis@scccd.edu}


\thanks{Partially supported by the National Science Foundation grant DMS-1247679}

\begin{abstract}
We define a two-player combinatorial game in which players take alternate turns; each turn consists on deleting a vertex of a graph, together with all the edges containing such vertex. If any vertex became isolated by a player's move then it would also be deleted. A player wins the game when the other player has no moves available. 

We study this game under various viewpoints: by finding specific strategies for certain families of graphs, through using properties of a graph's automorphism group, by writing a program to look at Sprague-Grundy numbers, and by studying the game when played on random graphs.

When analyzing Grim played on paths, using the Sprague-Grundy function, we find a  connection to a standing open question about Octal games.
\end{abstract}


\maketitle

In this article we define a two-person game played on the vertices of a graph, and then study it to find strategies for either player to win. The analysis of the game ends up depending heavily on the family of graphs considered at the time. This is why, in this article, we will consider a wide variety of tools from game theory, combinatorics and group theory, plus some programming, to attack this problem. In the following  section we will cover some basic notation and definitions that will be useful throughout the paper, as well as the actual game play.

\section{Introduction}

A \textit{graph} $G=(V,E)$ is a set $V$ of vertices and a set $E$ of edges connecting pairs of vertices. In this work we consider only graphs that are finite, simple and undirected.  As customary, the \textit{degree} of a vertex $v$, denoted $deg(v)$, is the number of edges that are incident with $v$. We say that a vertex $v$ is \textit{isolated} if $deg(v)=0$.  A graph is said to be connected if, given any two vertices $u,v \in V(G)$, there is a path in $G$ from $u$ to $v$.  Two graphs are said to be \textit{disjoint} when their vertex sets are disjoint; we will use the notation $G \cup  H$ for the graph formed by two disjoint graphs, $G$ and $H$.  The \emph{join of two graphs} $G$ and $H$,  denoted $G+H$,  is the graph with vertex-set $V(G+H) =V(G)\cup V(H)$ and edge-set $E(G+H) =E(G) \cup E(H) \cup \{vw; \ v\in V(G), \ w\in V(H)\}$. Finally, we will denote \emph{paths}, \emph{cycles} and \emph{wheels} by $P_n$, $C_n$, and $W_n$ respectively, where $n$ is the number of vertices in the graph, and $\overline{G}$ will denote the graph with $V(\overline{G})=V(G)$ and edges connecting only vertices that were not connected in $G$. More information about graphs may be found in \cite{BM}. 

In order to define our game  we re-phrase well-known notions in combinatorial game theory for a game played on graphs. For further reading about game theory, we direct the reader to \cite{WWMG}.

\begin{definition}
A  two-player game is said to be \textit{impartial} if the outcome of the game depends only on what player goes first. A  two-player game is said to be \textit{combinatorial} if both players have perfect information, there is no chance involved, the game ends after a finite number of moves, the game is impartial, and there is no draw. A  two-player game is said to be \textit{normal} if the last player to make a legal move is the winner.
\end{definition}

\begin{definition} \label{defNPposition}
Let $G$ and $H$ be graphs on which a combinatorial game can be played.
\begin{enumerate}
\item If $H$ is obtained from $G$ after a game move, then we will call $H$ a \emph{follower} of $G$ in the game.

\item If, given a graph $H$, there is a strategy to win for the next player making a move then we will say that $H$ is an \textit{$\mathcal{N}$ position} (the $\mathcal{N}$ is for next). If, given a graph $H$, the next player making a move does not have a strategy to win then $H$ is a  \textit{$\mathcal{P}$ position} (the  $\mathcal{P}$ is for previous).  
\end{enumerate}
\end{definition}

\begin{remark}
Not having a strategy to win means the other player has a strategy to win. This follows from the Sprague-Grundy Theorem, see \cite{G39} and \cite{S36}. So, if  $G$ is the starting graph and Player 1 will move first, then $G$ being an $\mathcal{N}$ position means Player 1 has  a strategy to win the game, but if $G$ is a $\mathcal{P}$ position then Player 2 has  a strategy to win the game. 
\end{remark}

Next we define our game, which we call \emph{Grim}.

\begin{definition}
Given a graph $H$, we define a \emph{legal move} of Grim on $H$ by a player  selecting and consequently deleting a vertex. When this vertex is deleted all edges adjacent to this vertex are also deleted, together with any other vertices (if any) that have become isolated because of the move. 

The game starts with Player 1 moving first on a pre-arranged, starting position, graph $G$ (naturally, if $G$ had any isolated vertices these would be deleted before the first player can move). After that, the two players alternate turns, making legal moves on the follower resulted from the previous player's move. They play until all vertices have been deleted. We refer to the winner of the game, or winner of the graph, as the player who makes the last legal move.
\end{definition}

\begin{remark}
Grim is a normal combinatorial two-player game. Given the nature of the game, we will also say it is a  vertex deletion game. 
\end{remark}

Note that changing the starting  graph may change the way Grim is played completely. So, in this article we are not just studying a specific game, but a family of them. This property is commonplace in the study of games played on graphs, of which there is a wide variety, see e.g. \emph{Node Kayles} in \cite{BK} and \cite{BS}, \emph{Connect-it} \cite{HR84}, \emph{Take Turn} \cite{S07}, and \cite{FS91}, \cite{K14} and \cite{NO05}.

Our interest in the subject of games played on graphs is due to the papers by Fukuyama (see \cite{F04} and \cite{F03}), where he studies Nim on graphs.

Our study of Grim is split as follows: In Section \ref{sec2} we look at weighted graphs and prove that an extension of Grim to these graphs is unnecessary, as the strategies that could be used to play Grim in weighted graphs are the same as those used to play Grim on certain unweighted graphs. In Section \ref{sec3} we find winning strategies for when Grim is played on complete and complete multipartite graphs. In Section \ref{sec4} we look at how certain symmetries in a graph could be exploited to guarantee a victory when Grim is played on such a graph. In Section \ref{sectionSG} we study paths, cycles, and other related graphs by using Sprague-Grundy functions; in this way we discover an interesting connection between Grim played on paths and an open problem about Octal games. Finally, in Section \ref{sec6} we consider random graphs in order to learn whether Player 1 has an overall better chance of winning at Grim.

\section{Weighted Graphs}\label{sec2}

It is not unusual to wonder whether a game played on graphs could be extended to be played on graphs with weighted vertices, maybe by thinking about the vertices of the graph as heaps of chips/tokens that can be removed one at a time by players. Of course, this way of playing Grim is similar to the way Nim is played, and thus it has been considered in the past (e.g. \cite{DR}). In this section we prove that the natural variation of Grim on weighted graphs is nothing but regular Grim on a different family of graphs.

When playing Grim on a weighted graph, we only allow a vertex with weight $t$ to be deleted after it has been selected $t$ times, or if it has been completely isolated. Thus, we can create many different games from the same graph by giving each vertex a new random weight.

As we studied these graphs, we discovered that we could replace a vertex with weight $t$ with $t$ `regular' vertices that do not share any edges. In order to prove such a claim the concept of a blowup of  a vertex becomes necessary.

\begin{definition}
Let $G$ be a graph with weighted vertices. Let $v\in V(G)$ having weight one and $t\in \mathbb{N}$. A \emph{$t$-blowup} of $v$ is an independent set $I_{v} =\{v_{1}, v_{2}, \ldots , v_{t}  \}$ of vertices of weight $1$ that `takes the place' of $v$. More precisely, wherever there was an edge joining $v$ to $w\in V(G)$ there is an edge joining $v_{j}$ with $w$, for all $j=1,\ldots, t$.

The graph obtained by the $t$-blowup of $v$ will be denoted $G(tv)$. Similarly, for $v,w\in V(G)$ and $s,t \in \mathbb{N}$ we denote a `double blowup' $G(tv)(sw)$ as $G(tv , sw)$. For multiple blowups we extend in the natural way the notation set of double blowups.
\end{definition}

Note that $G(1v)=G$, for all $v\in V(G)$.

\begin{lemma}\label{wv}
Let $t\in \N$. Let $G$ be a graph with weighted vertices, and let $v \in V(G)$ have weight $t$. We denote by $G(v)$ the $t$-blowup of the graph obtained by reducing the weight of $v$ to one. Then, the outcome of playing Grim on $G$ or on $G(v)$ is the same.
\end{lemma}

\begin{proof}
Let $I =\{v_1, \dots, v_t\}$ be the independent set used to obtain $G(v)$.  There are two possible ways to delete $v$ while playing on $G$, thus we look at the following two cases:  

\textit{Case 1}: Try to isolate $v$ in $G$. Using the same moves needed to remove all of $v$'s neighbors in $G$ we can isolate every vertex in $I$ (in $G(v))$. 

\textit{Case 2}: Removing $v$ in $G$ by repeatedly using it in the game. In order to remove $v$ from $G$, it has to be selected $t$ times. To fully remove $I$ from $G(v)$, each of the $t$ vertices in $I$ must be selected and removed. Thus removing $I$ from $G(v)$ requires an equivalent process as that needed to remove $v$ from $G$. 

Therefore, game play with $v$ or with $I$ will involve the same strategies. Finally, replacing $v$ with $I$ will not affect the other vertices in the graph, and thus it will not affect playing with them.
\end{proof}

An easy induction argument proves the main theorem of this section.

\begin{theorem}\label{npart}
Let $G$ be a weighted graph with $V(G) = \{v_1, \ldots, v_n\}$ and with $t_i$ being the weight of vertex $v_i$, for all $i=1, \ldots, n$. We denote by $G(v_1,v_2,\ldots, v_n)$ be the $(t_1, t_2,\ldots, t_n)$-blowup of the graph obtained by reducing (if possible) the weight of each vertex of $G$ to one. Then,  the outcome of playing Grim on $G$ or on $G(v_1,\ldots, v_n)$ is the same.
\end{theorem}

We conclude that Grim on weighted graphs does not need to be studied separately, and so from now on the game discussed is always Grim played on  unweighted graphs. 

Before we move on, we would like to let the the reader know that we will often use the word graph to mean the game played on that graph (as in `Player 1 wins graph $G$'). Most proofs in the following sections somehow describe the specific strategy to be used to win. 

\section{Complete and complete multipartite graphs}\label{sec3}

In this section we obtain specific strategies to win at Grim for complete graphs and complete multipartite graphs.  We start with a theorem about complete graphs. 

\begin{lemma}\label{thmcomplete}
Let $n \in \N$. Then, $K_n$ is an $\mathcal{N}$ position if and only if $n$ is even.
\end{lemma}

\begin{proof}
The case $n=1$ is immediate, as no game is ever played. For $n>1$, all moves are equivalent in a complete graph, and a move will always yield a complete graph, so the game always lasts $n-1$ moves. 
\end{proof}

Grim played on complete bipartite graphs will be analyzed next.

\begin{theorem}\label{thmcompletebiartire}
Let $m, n \in \N$. Then, the following hold.
\begin{enumerate}
\item $K_{1,n}$  is an $\mathcal{N}$ position.
\item Assume $m, n>1$. Then, $K_{m,n}$  is an $\mathcal{N}$ position if and only if $m+n$ is odd.
\end{enumerate}
\end{theorem}

\begin{proof}
The first claim is immediate. For the second claim we start by settling a base case.

For $K_{2,2}$, whatever move Player 1 makes, the follower is always $K_{1,2}$. By part 1, Player 2 wins this graph. 

Now we proceed by induction on $k=m+n$. Part 1 and the base case considered above tell us that we have the result for $k=3,4$. So, assume the result holds for $m+n=k\geq 4$. Consider $K_{a,b}$ where $a+b=k+1$. 

\textbf{Case 1:} If $k+1$ is odd, then $k+1\geq 5$. WLOG assume that $a>2$ and let Player 1 move by creating the follower $K_{a-1,b}$. Since both $a-1, b>1$ we use the induction hypothesis to determine that Player 1 wins $K_{a,b}$. 

\textbf{Case 2:} If $k+1$ is even, then $k+1\geq 6$. If WLOG $a=2$ then Player 1 would not want to move to create the follower $K_{1,b}$, as this would be won by Player 2. However, if they moved to create the follower $K_{2,b-1}$ then they would lose anyway, by induction, as $2,b-1 >1$. Now, if $a, b>2$ then whatever move Player 1 makes they would lose by induction, as $a-1, b-1\geq 2$. 
\end{proof}

In the rest of this section we extend Theorem \ref{thmcompletebiartire} to  most multipartite graphs. At this point we would like to remark that the only way to delete a vertex in a complete multipartite graph and, by doing so, isolate other vertices would be when we were making a move on $K_{1,n}$, which is a graph already studied in Theorem \ref{thmcompletebiartire}. This unique situation indicates that we should take a closer look at graphs that may have $K_{1,n}$ as a follower; $K_{1,1,n}$ is the only one of those graphs that has not been already studied in Theorem \ref{thmcompletebiartire}. Looking at the behavior of $K_{1,n}$ we obtained our next result.

\begin{lemma}\label{lemK11n}
Let $m,n\in \N$, and $G=K_{1,m,n}$. Then,
\begin{enumerate}
\item If $m=1$, then $G$ is an $\mathcal{N}$ position if and only if $n$ is even.

\item If $m=2$ and $n\geq 2$, then $G$ is an $\mathcal{N}$ position.

\item If $m,\ n\geq 2$, then $G$ is an $\mathcal{N}$ position if and only if $m+n$ is even.
\end{enumerate}
\end{lemma}

\begin{proof}
We start by noticing that the case $m=n=1$ is immediate. We will prove the case $m=1, \ n>1$ by induction on $n$. Our base cases are $K_{1,1, 1}$ and $K_{1,1, 2}$. We know that the first graph is a $\mathcal{P}$ position, and clearly, Player 1 wins $K_{1,1,2}$ by leaving the follower $K_{1,1,1}$ after their first move. For $n>2$, Player 1 should start by leaving $K_{1,1,n-1}$ as a follower (otherwise Player 2 would win the game immediately). Player 2 could leave $K_{1,n-1}$ or $K_{1,1,n-2}$ as a follower. In the first case, Player 1 wins $G$, and in the second case induction forces the result we wanted.

Let us assume that $m=2$ and $n\geq 2$. If $n$ is even then Player 1 leaves $K_{2,n}$ as a follower, which by Theorem  \ref{thmcompletebiartire} is won by Player 1. If $n$ is odd then Player 1 leaves $K_{1,1,n}$ as a follower, and thus they win $G$ by using the strategy discussed above.

Finally, assume $m,\ n \geq 3$. Since none of the players would like to leave $K_{1,2,k}$ as a follower then the winners of these graphs will alternate depending on $n$ being even or odd. Thus we look at the `base case' $K_{1,3,3}$, which is an $\mathcal{N}$ position by Theorem  \ref{thmcompletebiartire}, to get the desired result.
\end{proof}

The complete tripartite graphs `missing' from Lemma \ref{lemK11n}, and most of the other complete multipartite graphs are considered in the following theorem.

\begin{theorem}\label{thmno1s}
Let $G = K_{n_1,n_2,\cdots,n_{t}}$, where $t>2$ and $n_i \geq 2$, for all $i=1,\ldots, t$. Then, $G$ is a $\mathcal{P}$ position if and only if $|V(G)|$ is even.
\end{theorem}

\begin{proof}
Let $G = K_{n_1,n_2,\cdots,n_{t}}$, where $t>2$ and $n_i \geq 2$, for all $i=1,\ldots, t$. We define $S = \{n_i ; \ n_i \equiv 1\pmod2 \}$.

Suppose that $|V(G)|$ is even and $n_i \geq 2$, for all $i=1,\ldots, t$.  Note that $|S|$ is either $0$ or even.

Case 1: $|S|=0$, i.e. each partition of $G$ contains an even number of vertices. The strategy for Player 2 is as follows: the game starts with Player 1 deleting any vertex from the graph. Since the partition where this vertex was will have at least one more vertex in it, Player 2 deletes a vertex in it. In this way, Player 2 always has an available move in the same partition that Player 1 has chosen to play on. Hence, Player 2 is always the last person to delete a vertex from every partition of $G$, and thus they will always have the last move in the game, as they will be the only person able to play in the state $K_{1,n_i}$, when that time comes.

Case 2: $|S| =2k$, for some non-zero $k\in \N$. Now the strategy for Player 2 is two-fold: if Player 1 moves in an `even' partition of $G$ then Player 2 follows the strategy in Case 1, but if Player 1 deletes a vertex in an `odd' partition then Player 2 deletes a vertex in one of the other odd partitions (there would be an odd number of them available). Note that after one move by each player, the follower obtained would either be a graph that Case 1 considers, or $G$ would have  $|S| =2(k-1)$, and thus an induction argument would finish the proof. Note that the base case(s) of this induction could be track down to knowing who wins $K_{\text{odd, odd}}$ or $K_{\text{even, even}}$, which are both $\mathcal{P}$ positions by Theorem \ref{thmcompletebiartire}.

Now suppose that $|V(G)|$ is odd and $n_i \geq 2$, for all $i=1,\ldots, t$. Since there must be some  $n_i\in S$, the strategy for Player 1 is to delete any vertex in a partition with an odd number of vertices. In this way, they would leave a follower with an even number of vertices that also has every partition containing more than $1$ vertex. For this follower, Player 1 will move second and thus will win the graph. 
\end{proof}

We now look at a specific family of multipartite graphs allowing singletons as partitions. 

\begin{lemma}
Let $G = K_{1,n_{2},\ldots,n_{t}}$, where $3<t$, $1\leq k\leq t$, and  $n_i \geq 2$, for all $i=2,\ldots, t$. Then, $G$ is a $\mathcal{P}$ position if and only if $|V(G)|$ is even and $n_i>2$, for all $i$.
\end{lemma}

\begin{proof}
If $|V(G)|$ is odd then Player 1 leaves $K_{n_2,\ldots,n_{t}}$ as a follower, which is a $\mathcal{P}$ position (by Theorem \ref{thmno1s}), and thus Player 1 wins $G$. Now, if $|V(G)|$ is even then Player 1 would not want to leave $K_{n_2,\ldots,n_{t}}$ as a follower, as in this case, by Theorem \ref{thmno1s}, Player 2 would win $G$. Hence, Player 1 would leave WLOG $K_{1, n_2-1,n_3,\ldots,n_{t}}$ as a follower. If $n_2>2$ then by the previous case we would get that Player 2 wins $G$.  If we allow $r-1>0$ of the $n_i$ to be equal to $2$ then, after re-arrangement if needed, we get $G= K_{1, 2, \ldots 2, n_{r+1}, \ldots,n_{t}}$, where $n_{r+1},\ldots, n_t>2$. Note that $r< t$ because $|V(G)|$ is even. In this case, Player 1 leaves $K_{1, 2, \ldots 2, n_{r+1}-1, \ldots,n_{t}}$ as a follower, and thus Player 1 wins $G$ using the strategy discussed for the case $|V(G)|$ odd.
\end{proof}

There are several complete multipartite graphs that we have not studied, and we will not study in this paper; these, after a re-arrangement if needed, have the form $K_{1,\cdots, 1, n_{k+1},\ldots,n_{t}}$, where $3<t$, $1\leq k\leq t$, and  $n_i \geq 2$, for all $i=k+1,\ldots, t$. Some of them may be studied very easily, for example $K_{1,\ldots, 1}$ is a $\mathcal{P}$ position if and only if $t$ is odd. The conditions needed to understand others are much more complex. For instance, if $k>1$ and $k> \alpha_{k,t} = (n_{k+1}+\cdots + n_{t}) - (t-k)$, then $K_{1,\cdots, 1, n_{k+1},\ldots,n_{t}}$ is a $\mathcal{P}$ position if and only if $k$ and $\alpha_{k,t}$ have different parity. 

On the other hand, in order to understand the case $k\leq \alpha_{k,t}$ we need to look at these two parameters and also to \emph{how much} bigger than $k$ $\alpha_{k,t}$ is. 

Also, unlike most of the graphs studied in this section, $K_{1,n}$ and $K_{1,2,n}$ are always won by Player 1 independent on the parity of $n$. This situation is not that uncommon in the graphs we have not studied here, for instance the family of graphs of the form $K_{1,1,3,n}$, for $n\geq 3$, are also won by Player 1 by leaving the follower $K_{1,3,n}$ for $n$ even and $K_{1,1,2,n}$ for $n$ odd. There are several other families, similar to $K_{1,1,3,n}$, that may be proven to be won by Player 1 only.

All this suggests that a complete theory of dealing how Grim would play on multipartite graphs allowing singletons as partitions is very complex, and thus we plan to study it separately in a future article.

Other `standard' families of graphs can be studied in a similar fashion to what has been done in this section with complete and complete multipartite graphs. However, as it was seen in this section, and we will see again in Section \ref{sectionSG}, the analysis of Grim gains complexity quickly. Hence, in the following sections we will study Grim on graphs by using a less direct set of techniques.

\section{Automorphisms}\label{sec4}

Certain symmetries in a graph can be exploited to obtain strategies for winning at Grim. In this section we prove a couple of results of this type, and then we apply them to a few families of graphs.

\begin{definition}
A function $\sigma$ is said to be an \textit{automorphism} of a graph $G$ if $\sigma$ permutes the vertices of $G$ while preserving incidence. The group (under composition) of all the automorphisms of $G$ is denoted $Aut(G)$.
\end{definition}

\begin{theorem}\label{autoP2}
Let $G$ be a graph, and assume that $Aut(G)$ has an element $\sigma$ of order $2$ such that $v\neq \sigma(v)$ and $v\sigma(v) \notin E(G)$, for all $v\in V(G)$. Then, $G$ is a $\mathcal{P}$ position. 
\end{theorem}

\begin{proof}
Let $\sigma \in Aut(G)$ be such that $v\neq \sigma(v)$ and $v\sigma(v) \notin E(G)$, for all $v\in V(G)$.  We will prove that if Player 1 moves by selecting a vertex $v$, then Player 2 can always move by selecting $\sigma(v)$.

We will proceed by contradiction. Suppose that at a certain stage of the game, and for the first time, Player 1 moves by selecting a vertex $v$ leaving a follower that does not contain $\sigma(v)$. If $\sigma(v)$ was already deleted when Player 1 made their move, then either $\sigma(v)$ was deleted by Player 2, which would contradict that $v$ had not been deleted at this point, or $\sigma(v)$ was deleted by becoming isolated, which would force (as $\sigma \in Aut(G)$) $v$ to be isolated. In either case, the situation is impossible. So, the only option would be that $\sigma(v)$ gets isolated at the time of deleting $v$. But $v$ and $\sigma(v)$ are not adjacent, so deleting $v$ cannot isolate $\sigma(v)$.
\end{proof}

\begin{corollary} \label{autoP1}
Let $G$ be a graph with no isolated vertices, and assume that  $Aut(G)$ has an element $\sigma$ of order $2$ such that $v_0 = \sigma(v_0)$, for exactly one $v_0\in V(G)$, and $v\sigma(v) \notin E(G)$, for all $v\in V(G)$. Then, $G$ is an $\mathcal{N}$ position. 
\end{corollary}

\begin{proof}
Player 1 starts the game by deleting $v_0$, leaving a follower that is a $\mathcal{P}$ position by Theorem \ref{autoP2}. Hence, $G$ is an $\mathcal{N}$ position.
\end{proof}

The following corollary summarizes a few applications of Theorem \ref{autoP2} and Corollary \ref{autoP1}.

\begin{corollary}\label{coroddpathscycleswheels}
Let $n\in \N$. 
\begin{enumerate}
\item[\textbf{(1)}]   If $n$ is odd, and $n\geq 3$, then $P_n$ is an $\mathcal{N}$ position.
\item[\textbf{(2)}]   If $n$ is even, and $n\geq 4$,  then $C_n$ is a $\mathcal{P}$ position.
\item[\textbf{(3)}]   If $n$ is odd, and $n\geq 5$,  then $W_n$ is an $\mathcal{N}$ position.
\item[\textbf{(4)}]   $G\cup G$ is a $\mathcal{P}$ position, for all non-empty graphs $G$.
\end{enumerate}
\end{corollary}

\begin{proof}
\textbf{(1)} Use Corollary \ref{autoP1} with $\sigma$ the reflection across the middle vertex of $P_n$.  \\
\textbf{(2)} Modeling $C_n$ with a regular $n$-gon. Use Theorem \ref{autoP2} with $\sigma$ defined by reflecting every vertex across the center of the polygon. \\
\textbf{(3)} Use the same function in the proof of \textbf{(2)}, but now apply Corollary \ref{autoP1}, as the center of the wheel is fixed. \\
\textbf{(4)} Use Theorem \ref{autoP2} with $\sigma$ the identification between the two copies of $G$.
\end{proof}

Theorem \ref{autoP2} and Corollary \ref{autoP1} can also be used in the study of Cartesian products of graphs. 

\begin{definition}
Let $G$ and $H$ be two disjoint graphs. The \textit{Cartesian product} of $G$ and $H$, denoted $G\square H$, is defined by $V(G\square H)= V(G) \times V(H)$ and edges given by:
\begin{enumerate}
\item $(x,y)$ is adjacent to $(w,z)$ if $x=w$ and $y$ and $z$ are adjacent in $H$, or
\item $(x,y)$ is adjacent to $(w,z)$ if $y=z$ and $x$ and $w$ are adjacent in $G$.
\end{enumerate}
where $x,w \in V(G)$ and $y,z\in V(H)$.
\end{definition}

\begin{corollary}
Let $n>1$, $G_1,\cdots, G_n$ be graphs, and $G=G_{1} \square \cdots \square G_{n}$.
\begin{enumerate}
\item[\textbf{(1)}]   If $G_1,\cdots, G_n$ satisfy the hypothesis in Theorem  \ref{autoP2}, then, $G$ is a $\mathcal{P}$ position. 
\item[\textbf{(2)}]   If $G_1,\cdots, G_n$ satisfy the hypothesis in Corollary \ref{autoP1}, then $G$ is  an $\mathcal{N}$ position.
\end{enumerate}
\end{corollary}

\begin{proof}
\textbf{(1)}  Let $\sigma_i \in Aut(G_i)$ be the automorphism satisfying the hypothesis in Theorem  \ref{autoP2}. It is easy to see that $\sigma$ defined by
\[
\sigma(v_1, \ldots, v_n) = (\sigma_1(v_1),  \ldots , \sigma_n(v_n) )
\]
is an automorphism of $G$ also satisfying the hypothesis in Theorem  \ref{autoP2}. The result follows. \\
The proof for \textbf{(2)}  uses the very same ideas, so we omit it.
\end{proof}

An attempt to find strategies that allow us to determine who would win $G\square H$, for arbitrary graphs $G$ and $H$, was made but it failed because we realized that keeping track of all the possible avenues the game could take was an impossible task.  A natural approach to how to `keep track' of what happens in a game is discussed in the following section.

\section{The Sprague-Grundy Function}\label{sectionSG}

One of the most useful tools in the study of combinatorial games is the Sprague-Grundy function. We proceed to define it, in the context of Grim, and then we focus our efforts on finding its values  for paths.

\begin{definition}
Let $\N_0 = \N \cup \{0\}$ and $G$ be a graph. We define $\mathcal{F}(G)$ to be the set of all followers of $G$ (in the game Grim). The \emph{Sprague-Grundy function} of $G$ is a function $\mathcal{SG}: V \to \N_0$, defined recursively as follows
\[ 
\mathcal{SG}(G) = \min\{n \in \N_0 ; \ n \neq \mathcal{SG}(H), \text{ for all } H \in \mathcal{F}(G) \}.
\]
If we define the minimal excludant, or mex, of a set of non-negative integers as the smallest non-negative integer not in the set, then we may write
\[
\mathcal{SG}(G) =\mex\{\mathcal{SG}(H) ; \  H \in \mathcal{F}(G) \}.
\]
\end{definition}

\begin{remark}\label{rem0or1}
$\mathcal{SG}(G)= 0$ if and only if $G$ is a $\mathcal{P}$ position.
\end{remark}

The process to calculate the Sprague-Grundy values by hand is laborious, so we wrote a program in Visual FoxPro to compute the Sprague-Grundy value of Grim on paths. We focused on paths because in Corollary \ref{coroddpathscycleswheels} we learned that odd paths are $\mathcal{N}$ positions, but when wanting to look at even paths no pattern seemed to exist.

Using our program, we were able to compute all the values of $\mathcal{SG}(\mathcal P_n)$\footnote{A library of these values may be obtained at  www.gamecalledgrim.com.} for up to $n = 10^7$. Using these results we know that if $n\leq 10^7$, then $\mathcal{SG}(\mathcal P_n)=0$ only when 
\[
n \in \{4, 12, 20, 30, 46, 72, 98, 124, 150, 176, 314, 408\},
\]
which is something \L\c acko and \L\c acki (unpublished work, see \cite{LL}) had also determined (although they consider $\mathcal P_1$ a $\mathcal{P}$ position while we do not consider $\mathcal P_1$ a graph where we would play Grim, as all isolated vertices in a graph are deleted before a game starts). Moreover,  they also claim that  $\mathcal{SG}(\mathcal P_n)\neq 0$ for all values in the range $10^7\leq n\leq 10^8$. The obvious question, at this point, is whether or not there are any other even paths with $\mathcal{SG}(\mathcal P_n)=0$. It turns out that the answer to this question is related to a different open problem in game theory.

\begin{definition}
An \emph{octal game} is a normal impartial game played on heaps of chips. A move consists of taking a number of chips from a single heap, eliminating them, and then redistributing the remaining chips in the heap (if any left) into 0 (all chips in the heap were eliminated), 1 (leave the remaining chips intact), or 2 heaps; this number is fixed in advance for the game and depends only on how many chips were eliminated in the move. Given that there are several variables to determine at the beginning of the game, the rules of the game are encoded into a string of digits as follows: \\
If in a move we remove $k$ chips from a heap and we are allowed to break the remaining ones into $a$ heaps then we say that $\chi_k(a)=1$. In the case that we were not allowed to break the remaining ones into $a$ heaps we say that $\chi_k(a)=0$. The code of the game is $d=0.d_1d_2\ldots$, where
\[
d_k=\chi_k(0) 2^0+ \chi_k(1) 2^1+ \chi_k(2) 2^2.
\]

It follows that, since $6=2^{1}+2^{2}$, the game \emph{Octal $.6$} is a game in which every move consists of one chip being removed from a heap and what remains of that heap must be left in exactly 1 or 2 non-empty heaps.
\end{definition}

Several octal games have Sprague-Grundy sequences that after a while become periodic (see \cite{FOctal}, for example); it is not know whether Octal $.6$ has a period. This is relevant to us because the terms in the sequences of Sprague-Grundy numbers for Grim on paths and Octal $.6$ are the same  (see Theorem \ref{thmoctalSGgrimpaths} below). Thus our question about values of $n$ for which $\mathcal{SG}(\mathcal P_n)=0$ becomes a problem about when Octal $.6$ is won by Player 2, which is an open problem (e.g. see \cite{CGGN}). 

In order to create some intuition on the reader about our next result, we notice that we can identify removing a chip from a heap in Octal $.6$ to deleting a vertex in a path (under Grim rules). More explicitly, depending on what vertex we delete, we always leave one or two paths left; these two paths would be the equivalent of the two heaps left after removing a chip in Octal $.6$.

\begin{theorem}\label{thmoctalSGgrimpaths}
The Sprague-Grundy sequence of values of Grim played on paths is equivalent to the Sprague-Grundy sequence of values of Octal $.6$.
\end{theorem}

\begin{proof}
Let $n\in \N$, $\mathcal{O}_n$ be Octal $.6$ played on a heap of size $n$, and $\mathcal P_n$ be Grim played on a path on $n$ vertices. We will prove that $\mathcal{SG}(\mathcal P_n) = \mathcal{SG}(\mathcal{O}_n)$. We will do this by induction on $n$.

For $n=1$ we get that both games are trivial, and thus $\mathcal{SG}(\mathcal P_1)=\mathcal{SG}(\mathcal O_1)=0$. We now assume that  $\mathcal{SG}(\mathcal P_i)=\mathcal{SG}(\mathcal O_i)$, for all $i\leq n$.

First we notice that the followers of $\mathcal P_{n+1}$ and $\mathcal{O}_{n+1}$ are:
\begin{eqnarray*}
\mathcal{F}(\mathcal P_{n+1}) & = &\{ \mathcal P_n,  \mathcal P_{n-1} \cup \mathcal P_1, \mathcal P_{n-2} \cup \mathcal P_2, \ldots\} \\
\mathcal{F}(\mathcal O_{n+1}) & =& \{\mathcal O_n, \mathcal O_{n-1} \cup \mathcal O_1, \mathcal O_{n-2} \cup \mathcal O_2, \ldots \}.
\end{eqnarray*}

Using Nim Sum (see \cite{WWMG}), denoted $+_N$ below, to evaluate the Sprague-Grundy number of the followers of $\mathcal P_{n+1}$ and $\mathcal{O}_{n+1}$, we get:
\begin{eqnarray*}
\mathcal{SG}(\mathcal P_{n+1}) & = & \mex\{\mathcal{SG}(\mathcal P_n), \mathcal{SG}(\mathcal P_{n-1})+_N\mathcal{SG}(\mathcal P_1), \mathcal{SG}(\mathcal P_{n-2})+_N\mathcal{SG}(\mathcal P_2), \ldots \} \\
\mathcal{SG}(\mathcal O_{n+1}) & = & \mex\{\mathcal{SG}(\mathcal O_n), \mathcal{SG}(\mathcal O_{n-1})+_N\mathcal{SG}(\mathcal O_1), \mathcal{SG}(\mathcal O_{n-2})+_N\mathcal{SG}(\mathcal O_2), \ldots \}.
\end{eqnarray*}

Using the inductive hypothesis and Nim Sum we get that 
\begin{eqnarray*}
\mathcal{SG}(\mathcal P_n)  &= & \mathcal{SG}(\mathcal O_n) \\
\mathcal{SG}(\mathcal P_{n-1})+_N \mathcal{SG}(\mathcal P_1) &=& \mathcal{SG}(\mathcal O_{n-1})+_N \mathcal{SG}(\mathcal O_1) \\
&  \vdots &
\end{eqnarray*}

Hence, $\mathcal{SG}(\mathcal P_{n+1})=\mathcal{SG}(\mathcal O_{n+1})$.
\end{proof}

If we knew all the values in the sequence $\{\mathcal{SG}(\mathcal P_n)\}_{n=1}^{\infty}$ we would also know several other Sprague-Grundy sequences. We present two results of this type, without a proof, in the following corollaries.

\begin{corollary}
The Sprague-Grundy sequence for Grim on cycles is given by: 
\[
\mathcal{SG}(\mathcal C_n) = 
\begin{cases} 1 & \text{if } \mathcal{SG}(\mathcal P_{n-1})=0 \\ 
0 & \text{otherwise}
\end{cases}
\]
In particular, $\mathcal C_n$ is a $\mathcal{P}$ position when $n$ is even.
\end{corollary}

Finally, we get a result about wheels that depends on whatever knowledge we had about paths.

\begin{definition}
A \emph{wheel graph} on $n$ vertices is defined by $\mathcal{W}_n=\mathcal{C}_{n-1}+K_1$. $K_1$ is called the center of the wheel.
\end{definition}

\begin{corollary}
Let $n\geq 4$, then $\mathcal{W}_n$ is an $\mathcal{N}$ position if and only if $\mathcal{P}_{n-2}$ is an $\mathcal{N}$ position. 
\end{corollary}

\begin{proof}
$\mathcal W_n$ has only two followers: $\mathcal{C}_{n-1}$, obtained from deleting the center of the wheel and  $\mathcal{W}'_{n-1}$, obtained by deleting a vertex different from the center of $\mathcal W_n$.

If Player 1 wins $\mathcal{P}_n$, then they can win $\mathcal{W}_{n+2}$ by removing the center vertex and leaving $\mathcal{C}_{n+1}$ as a follower, this then forces Player 2 to leave $\mathcal{P}_n$ as a follower, which is an $\mathcal{N}$ position.  

If Player 2 wins $\mathcal{P}_n$, then since $\mathcal{W}_{n+2}$ has only two followers and both of them are one deleted vertex away from $\mathcal{P}_n$, Player 2 can leave $\mathcal{P}_n$ as a follower after their first move. 
\end{proof}

\section{Grim on Random Graphs}\label{sec6}

We are interested in understanding what happens when Grim is played on a random graph. Our main reason is that evidence in previous sections has shown that the game seems to favor Player 1 overall. In order to do this, we will use Erd\H os-R\'enyi random graphs. 

\begin{definition} 
Let $n\in \N$ and $0 \leq p \leq 1$. An \emph{Erd\H os-R\'enyi random graph} $\mathcal{G}(n,p)$ is defined by having $n$ vertices and, for any pair of vertices, the existence of an edge connecting them has probability $p$.
\end{definition}

In this section, we will work on identifying probabilities $p$ for which Player 1 and Player 2 have equal chances of winning a random graph $\mathcal{G}(n,p)$.

\begin{remark} \label{lemmap=1}
The probability of \emph{any} given graph on $n$ vertices and $k$ edges to exist is given by $p^{k}(1-p)^{\binom{n}{2}-k}$. So, if $p=0.5$, then all possible graphs on $n$ vertices have equal probability to exist; this value being $(1/2)^{{n \choose2}}$. 
\end{remark}

\begin{definition} 
The \emph{winning probability} of Player 1, denoted $W_1(p)$, is a function that describes the probability for Player 1 to win on a random graph, for a fixed $n$ and a given edge probability $p$. It is of the form 
\[
W_1(p)=\sum_{\mathcal{G}} X(\mathcal{G})P(\mathcal{G}),
\]
where the sum runs through all graphs (including the empty graph, which is a $\mathcal{P}$ position), $P(\mathcal{G})=p^k(1-p)^{{n \choose 2}-k}$ for a graph on $k$ edges, and
\[
X(\mathcal{G}) =
\begin{cases}
1 & \text{if } \mathcal{SG}(\mathcal{G})\neq 0 \\
0 & \text{if } \mathcal{SG}(\mathcal{G}) = 0
\end{cases}
\]
\end{definition}

The \emph{winning probability} of Player 2, denoted $W_2(p)$, may be defined similarly.

As an illustration, let us consider Erd\H os-R\'enyi random graphs on $n=3$ vertices with edges generated with probability $p$. We know that there are eight $3$-vertex graphs. Out of these eight graphs, Player 1 would win all three graphs with one edge, all three graphs with 2 edges, and lose the one complete graph and the empty graph. Hence, the winning probabilities are:
\[
W_1(p)  =  3p(1-p)^2+3p^2(1-p)  \hspace{1.5in}  W_2(p)  =  (1-p)^3+p^3.
\]

Note that setting $W_1(p)  =  W_2(p)$ yields the quadratic equation $6p^2- 6p+1=0$, which has solutions
\[
p = \frac{3\pm \sqrt{3}}{6};
\]
they are approximately $0.21$ and $0.79$.

Thus, Player 1 and Player 2 have an equal chance of winning on a 3-vertex graph if the chance of each edge appearing is about $21\%$ or $79\%$. Moreover, looking at the graph of $W_1(p)  - W_2(p)$ we get that Player 2 has more than a $50\%$ chance of winning as long as 
\[
p< \frac{3- \sqrt{3}}{6} \hspace{.5in}  \text{or}  \hspace{.5in}  \frac{3 + \sqrt{3}}{6} < p.
\]

On the other hand, if we now consider Erd\H os-R\'enyi random graphs on $n=4$ vertices. We get  that the winning probability for Player 2 is 
\[
W_2(p)=3p^2(1-p)^4+16p^3(1-p)^3 + (1-p)^5,
\]
which is equal to $0.5$ for $p\sim 0.16$, and always below $0.5$ after.

Our examples suggest the following theorem.

\begin{theorem}\label{thmp_00.5}
Let $n$ be odd. For an Erd\H os-R\'enyi random graph $\mathcal{G}(n,p)$, there exists $0<p_0<1$ such that $W_2(p)\geq0.5$, for all $p\geq p_0$.
\end{theorem}

\begin{proof} 
We will get a lower bound for $W_2(p)$ and use it to derive the existence of $p_0$. Note that, by  Remark \ref{lemmap=1}, we get
\[
W_2(p)\geq p^{{n \choose2}}.
\]

We set this bound equal to 0.5 and solve for $p$. We get the solution
\[
p_0  =  \left(\frac{1}{4}\right)^{\frac{1}{n^2-n}},
\]
which is in $(0,1)$, for every $n$.

Since $p_0 < p$ implies $0.5 = p_0^{{n \choose2}}< p^{{n \choose2}}$, we get that $W_2(p)\geq0.5$, for all $p\geq p_0$.
\end{proof}

Notice that as the number of vertices gets large, the bound for the value of $p_0$ given in the proof of Theorem \ref{thmp_00.5} approaches $1$ and thus, overall, Player 1 has increasingly better chances to win. It would be good to find an explicit expression for the least possible $p_0$, but this seems to be a really hard problem, as we do not know which graphs are $\mathcal{P}$ positions.

\section{Acknowledgments}

The authors would like to thank the referee for his/her thorough review, and for his/her comments and suggestions; they have definitely improve the flow in the exposition, the presentation, and the overall quality of this article.

\section{Conflict of Interests}

The authors declare that there is no conflict of interest regarding the publication of this article.


\end{document}